\documentclass[11pt]{article}
\usepackage{amsmath}
\usepackage{amsfonts}
\usepackage{amssymb}
\usepackage{amsthm}

\newtheorem{theorem}{Theorem}
\newtheorem{lemma}[theorem]{Lemma}

\newtheorem{proposition}[theorem]{Proposition}

\newtheorem*{SIP}{\emph{Stable Interpolation Property}}

\newtheorem*{kf}{Approximation Property}
\newtheorem{claim}[theorem]{Claim}

\newcommand{\supp}[1]{{\text{\rm supp}}({#1})}

\newcommand{\ints}{\mathbb{Z}}
\newcommand{\Z}{\mathbb{Z}}

\newcommand{\nats}{\mathbb{N}}
\newcommand{\reals}{\mathbb{R}}

\newcommand{\comps}{\mathbb{C}}

\newcommand{\dif}{\, \mathrm{d}}

\renewcommand{\d}{\mathrm{dist}}
\renewcommand{\k}{k_0}

\newcommand{\calf}{\mathcal F}
\newcommand{\calg}{\mathcal G}
\newcommand{\bfX}{\mathbf X}

\newcommand{\mult}[1]{\mathcal{M}_{#1}}

\newcommand{\pw}{\mathrm{PW}}

\newcommand{\cl}[1]{\mathcal{#1}}
\newcommand{\vp}{\varepsilon}

\newcommand{\sig}{\boldsymbol{\sigma}}
\newcommand{\1}{{\rm 1\hspace*{-0.4ex}%
\rule{0.1ex}{1.52ex}\hspace*{0.2ex}}}
\numberwithin{equation}{section}
\numberwithin{theorem}{section}

\title{Cardinal Interpolation with Gaussian Kernels \thanks{\emph{2000 Subject Classification:}41A15,41A25,41A63, 42B15}
}
\author{T. Hangelbroek\thanks{Department of Mathematics, 
Texas A\&M University, College Station, TX 77843, USA. Research supported
    by NSF Postdoctoral Research Fellowship.},
W. Madych\thanks{ Department of Mathematics, University of Connecticut Storrs, CT 06269, USA.},
F. Narcowich\thanks{Department of Mathematics, Texas A\&M University, College Station, TX 77843, USA. Research supported by grant DMS-0807033 from the National Science Foundation.},
J.~D.~Ward\thanks{Department of Mathematics, Texas A\&M University, College Station, TX 77843, USA.
Research supported by grant DMS-0807033 from the National Science Foundation.}
}

\begin{document}


\maketitle
\begin{abstract}
  In this paper, interpolation by scaled multi-integer translates of
  Gaussian kernels is studied. The main result establishes $L_p$
  Sobolev error estimates and shows that the error is controlled by
  the $L_p$ multiplier norm of a Fourier multiplier closely related to
  the cardinal interpolant, and comparable to the Hilbert transform.
  Consequently, its multiplier norm is bounded independent of the grid
  spacing when $1<p<\infty$, and involves a logarithmic term when
  $p=1$ or $\infty$.
\end{abstract}
\section{Introduction}

We consider interpolation by means of linear combinations of
translates of a fixed Gaussian where the data consists of samples
$f(hk)$, $k$ in $\Z^n$, of a continuous functions $f$ and derive $L_p$
error estimates in terms of $h$ and appropriate smoothness properties
of $f$. Namely, our interpolants are of the form
$$
s_h(x) = \sum_{\k\in \ints^n} a_k e^{-|x- hk|^2}
$$
and, roughly speaking, our estimates are of the form
$\|f-s_h\|_{L_p} \leq Ch^k \|f\|_{W_p^k}$, valid for 
sufficiently large $k$. 

In a series of papers \cite{S,RS1,RS2,RS3}, Riemenschneider and
Sivakumar have developed a comprehensive theory of cardinal
interpolation by Gaussians, treating issues of existence/uniqueness of
interpolants, decay of fundamental functions, bounds on Lebesgue
constants, and $L_p$ stability for data in $\ell_p$.  This is an
outgrowth of a general theory of cardinal interpolation that started
for univariate splines with Schoenberg \cite{Scho}, was extended to
several variables for {\em box splines} by de Boor, H\"{o}llig, and
Riemenschneider \cite{BoHR}, and for radial basis functions by Madych
and Nelson \cite{MN}, and Buhmann\cite{Buh}.

One important topic not addressed by Riemenschneider and Sivakumar is
the rate of convergence of the interpolant to a given smooth target
function.  In general, global approximation of smooth, non-analytic
functions with Gaussians and other $C^{\infty}$ positive definite
functions poses a considerable challenge.  In these cases, rates of
decay are often known to be {\em spectral} \footnote{decaying
  exponentially fast} \cite{MadNe}, but typically hold only for target
functions that are infinitely smooth.  More mainstream, linear error
estimates for functions of finite smoothness (from Sobolev spaces, for
example), where the $L_p$ norm of the error decays at a rate governed
by the $L_p$ smoothness of the target function\footnote{This stands in
  contrast to {\em nonlinear} approximation, where the $L_p$ norm of
  the error decays at a rate dependent on a different smoothness norm,
  but generally Gaussians at different scales must be employed,
  cf. \cite{HR, KP}} have been more elusive.  In this regard, the
approximation power of the underlying spaces has only been thoroughly
understood in the shift invariant setting, but that of the
interpolants has not as yet been studied. In the case of Gaussians,
the $L_2$ error estimates fall under the shift invariant theory
developed in \cite{BDR1,BDR2} and
generalized to $L_p$ by Johnson \cite[Section 4]{Johnson}.

\subsection{Overview} 
In this article, we demonstrate convergence rates for cardinal
Gaussian interpolation for target functions having finite smoothness.
The basic strategy we adopt has been developed in \cite{NW} and is
based around a $K$-functional argument, usually brought about by
band limiting the target function in a precise way.  The techniques we
use for estimating the error involve showing that interpolating by
band-limited functions delivers precise approximation rates, and that
such band-limited interpolants form a very useful class of target
functions on which cardinal Gaussian interpolation is very stable.  In
this case, the main challenge is to demonstrate this extra stability;
this is accomplished by carefully controlling the multiplier norm of
the Lagrange function.

In Section \ref{main_result}, we describe in detail our main result
and the strategy that we will use to obtain it. At the end of that
section, we discuss a way to generalize the results to other,
analogous situations. The approximation results for interpolation by
band-limited functions is the focus of Section \ref{bandlimited}.  The
main tools for cardinal Gaussian interpolation and the key multiplier
estimate are given in Section \ref{CIG}.  The extra stability results
are demonstrated in Section \ref{stable_interpolation}.

\subsection{Notation and Background}
The symbol $C$, often with a subscript, will always represent a
constant. The subscript is used to indicate dependence on various
parameters.  The value of $C$ may change, sometimes within the same
line.

Let $\mathcal{S}$ denote the space of Schwartz functions on
$\reals^n$.  The $n$-dimensional Fourier transform is given
by 
$\widehat{f}(\xi) = \int_{\reals^n} f(x) e^{-i\langle\xi,x\rangle}
\dif x$, and its inverse is $ f^{\vee}(x) = (2\pi)^{-n}
\int_{\reals^n} f(\xi) e^{i\langle x,\xi\rangle} \dif \xi.$ An
important property of the Gaussian functions
\begin{equation}\label{Gaussian}
\mathsf{g}: x \mapsto \exp\bigl[-|x|^2\bigr],
\end{equation}
is that they satisfy
$\widehat{\mathsf{g}}   = \pi^{n/2} \mathsf{g}(\cdot/2).$

Given a {\em multiplier} $m:\reals^n \to \comps$, which maps $f$ 
to $\left( \widehat{f}  \cdot m \right)^{\vee}$
the $L_p\to L_p$ operator norm is denoted,
$$
\|m\|_{\mult{p}} :=\sup_{\|f\|_p=1} \left\| \left(\widehat{f}\cdot
    m\right)^{\vee}\right\|_p
$$ 

Let $\Omega \subseteq \reals^n$ be a domain that satisfies a uniform
cone condition. The Sobolev space $W_p^k(\Omega)$ is endowed its usual
seminorm $|\cdot|_{W_p^k(\Omega)}$ and norm $\| \cdot
\|_{W_p^k(\Omega)}$, defined by
$$
|f|_{W_p^k(\Omega)}:= \sup_{|\alpha|=k} \|D^{\alpha} f\|_{L_p(\Omega)},
\quad \text{and} \quad 
\|f\|_{W_p^k(\Omega)}:= \|f\|_{L_p(\Omega)} + |f|_{W_p^k(\Omega)}.
$$
When just the seminorm is used, the resulting space is a
\emph{Beppo-Levi} space. We will denote it by $\dot
W_p^k(\Omega)$. Finally, most of the time we will be dealing with
$\Omega=\reals^n$. When that is the case, we will just use $L_p$,
$W_p^k$, or $\dot W_p^k$.

We denote by $B(x,r)$ the ball in $\reals^n$ having center $x$ and
radius $r$.  The space of entire functions of exponential type, viewed
as tempered distributions whose Fourier transform is supported in
$B(0,R)$, is given by
$$
\pw(R) := \{ f\in \mathcal{S}'\mid \supp{\widehat{f}}
\subset B(0,R)\}.
$$
This is the \emph{Paley-Wiener space} of band-limited entire
functions. In addition, we let $\pw_p^k(R) := W_p^k\cap \pw(R)$, the
band-limited entire functions in $W_p^k$.

\section{The Main Result}\label{main_result}
We consider interpolation of a continuous function $f:C(\reals^n)\to
\comps$ at gridded centers $h \ints^n$ using elements of the linear
span of shifts of a fixed Gaussian kernel $g(x) = \exp(-|x|^2)$, the
span being closed in the topology of uniform convergence on compact
sets. In other words, we consider interpolation by functions of the
form $s_{f,h}(x)=\sum_{\zeta \in h\ints^n} a_{\zeta} \mathsf{g}(x -
\zeta).$

The existence and uniqueness of the interpolant 
is a consequence of the existence/uniqueness 
of the {\em Lagrange} function $\chi_{h}$. Indeed, 
$$
I_h f(x) := \sum_{\xi\in h \ints^n} f(\xi) \chi_{h}(x-\xi).
$$
The Lagrange function is the function in the (extended) span of shifts
of $g$ equaling $1$ at the origin and $0$ at all other dilated
multi-integers.  These have been studied in \cite{RS1,RS2}, as the
{\em cardinal} interpolant: $L_{\lambda}^{[n]} (y)= \sum_{j\in
  \ints^n} c_j \exp(- \lambda |y- j|^2)$. The relation between the two
is
\begin{equation}\label{Card-Lagrange}
  \chi_{h}(x) = \sum_{\xi\in h\ints^n} b_{\xi} \mathsf{g}(x-\xi) = 
  L_{h^2}^{[n]}\left(\frac{x}{h}\right).
\end{equation}
The problem that we have set forth for ourselves is to obtain a good
understanding of how well Gaussian interpolants approximate functions
in various Sobolev spaces. Our main result is the following:

\begin{theorem}[\bfseries{Main Result}]\label{main_thm}
  Let $1< p< \infty$ and $k> n/p$. There exists a constant $C_p$ so
  that for $f\in W_p^k(\reals^n)$, the Gaussian interpolant $I_h f =
  \sum_{\xi \in h\ints^n} f(\xi) \chi_h(\cdot-\xi)$ satisfies
$$\|f -I_h f\|_p \le C_p h^k \|f\|_{W_p^k}.$$
For $p=1$ and $k\ge n$ or $p=\infty$ and $k>0$, there is a constant
$C$ so that for $f\in W_p^k(\reals^n)$
$$\|f -I_h f\|_p \le C  (1+|\log h|)^n h^k \|f\|_{W_p^k}.$$
\end{theorem}

The strategy for proving the main theorem involves two steps, which we
will discuss before we give the proof. The first is showing that
interpolation of functions in $\dot W_p^k$ by band-limited functions
on $h\Z^n$ is possible, and that the band-limited interpolants
approximate functions in $\dot W_p^k$ very well. This is the content
of the Approximation Property described in Lemma
\ref{BandlimitedIntLemma} below. We will prove this in Section
\ref{bandlimited}.

\begin{lemma}[\bfseries{\emph{Approximation
    Property}}] \label{BandlimitedIntLemma} Let $b = \pi+\vp$ with
  $0<\vp<\pi.$ If $f\in \dot W_p^k$ for $k>n/p$, $1\le p\le
  \infty$ then given $h>0$ there is a function $\tilde f$ satisfying
\begin{equation}\label{eq1}
\tilde f\in \pw(b/h)
\end{equation}
such that
\begin{align}
\label{eq2}
 &\tilde f|_{h\ints^n} = f|_{h\ints^n}\\
\label{eq4}
 &\|f-\tilde f\|_{L_p} \le Ch^k |f|_{W_p^k}\\
\label{eq3}
&|\tilde f|_{W_p^k} \le C|f|_{W_p^k}
\end{align}
\end{lemma}

The second step is to show that Gaussian interpolation is stable on
the space $\pw_p^k(b/h)$, endowed with the Sobolev norm. We will prove
this in Section \ref{stable_interpolation}.
\begin{SIP}
  Let $k\in\nats$, $k>n/p$.  We say that the interpolation operators
  $I_h$ satisfy the {\em stable interpolation property} on the
  family $\pw_p^k(b/h)$ if there is
  $Q_p:h\to (0,\infty)$ so that one has
$$
|I_h f|_{W_p^k} \le Q_p(h) \|f\|_{W_p^k}
\quad\text{ for all }\ 
f\in \pw_p^k(b/h).
$$
\end{SIP}

\begin{proposition}\label{gaussian_pw_error}
  Let $1\le p\le \infty$. If the interpolation operators $I_h$ satisfy
  the Stable Interpolation Property on $\pw_p^k(b/h)$, then for
  $k>n/p$,
$$
\|I_h f- f\|_p \le C (1+ Q_p(h)) h^k \|f\|_{W_p^k}\,,\ f\in W_p^k.
$$
\end{proposition}

\begin{proof}
  We note that, by the Stable Interpolation Property, $\|f - I_h f\|_p
  \le \|f -\tilde f\|_p + \|\tilde f - I_h \tilde f\|_p$, since $I_h
  \tilde f = I_h f$.  By the Approximation Property, the error
  $\|\tilde f-f\|_p $ is controlled by $h^k$. Thus, the analysis of
  interpolation error $\tilde f- I_h \tilde f$ reduces to
  investigating its behavior on $h\Z^n$, where the error vanishes. An
  estimate on the size of a smooth function having many zeroes was
  proved by Madych \& Potter \cite[Corollary 1]{MadPo}. Employing it,
  we obtain
$$ 
\|\tilde f -I_h\tilde f\|_p 
\le 
C h^k |\tilde f -I_h\tilde f|_{W_p^k}
\le 
C h^k \left(| \tilde f|_{W_p^k} + | I_f \tilde f|_{W_p^k}\right).
$$
Thus, the interpolation error is controlled entirely by the norm of
$I_h$ as an operator from $W_p^k$ to $W_p^k$ for $k>n/p$. Invoking
the stable interpolation property, we obtain
\begin{eqnarray*}
  \|f-I_hf\|_p 
  &\lesssim&   
  \|f - \tilde f\|_{p} +   \|\tilde f- I_h \tilde f\|_p\\
  &\lesssim&
  h^k \|f\|_{W_p^k} +   h^k \left(\|\tilde f\|_{W_p^k} + Q_p(h) 
    \| \tilde f\|_{W_p^k}\right)\\
  &\lesssim& 
  h^k \bigl(1+ Q_p(h) \bigr)\|f\|_{W_p^k}.
\end{eqnarray*}
\end{proof}

\begin{proof}[Proof of Theorem \ref{main_thm}] Finishing the proof
  only requires only showing that the Stable Interpolation Property
  holds, with the appropriate $Q_p(h)$. We do this in Lemma
  \ref{StableIntLemma}. There we also show that 
\[
Q_p(h) \le 
\left\{
\begin{array}{cc}
C_p& 1<p<\infty,\\
C(1+|\log h |)^n & p=1,\infty.
\end{array}
\right.
\]
Apart from $C_p$ being explicitly dependent on $p$, the two constants
depend on $n$, $k$, and the choice of the parameter $b$. Using this
estimate on $Q_p$ in Proposition~\ref{gaussian_pw_error} then yields
the result.

\end{proof}

\paragraph{Generalizations}
The interpolation problems considered here specifically involve only
spaces of band-limited functions and spaces of Gaussians. However, it
is worthwhile to broaden the context and describe these problems in a
more general way.

Let $\Xi\subset \Omega$ be a discrete set of points, which we will
call \emph{nodes} or \emph{centers}, which will play the role of
$h\Z^n$ above as sites for interpolation. Since $\Xi$ doesn't have to
lie on a grid, we will describe how dense $\Xi$ is in $\Omega$ using
the {\em fill distance} or \emph{mesh norm}, which is defined by
$h(\Xi):=\sup_{x\in \reals^n} \d(x,\Xi) $. Normally, one would not be
dealing with all possible $\Xi$, but rather with a specific class of
sets, $\mathbf{X}$.

Suppose that for each $\Xi$ in $\bfX$ there are two spaces of
functions, $\calf_\Xi$ and $\calg_\Xi$. These form families
$\calf_\bfX =\{\calf_\Xi\mid \xi \in \bfX\}$ and $\calg_\bfX
=\{\calg_\Xi\mid \xi \in \bfX\}$. Here, $\calf_\xi$ and $\calg_\Xi$
are analogous to the band-limited functions $\pw(b/h)$ and the
Gaussians, respectively. Both are contained in $C(\Omega)\cap \dot
W_p^k(\Omega)$. In addition, we will assume that for each $\calg_\Xi$
there is an interpolation operator $I_{\calg_{\Xi}}: C(\Omega)\to
\calg_\Xi$. Finally, we suppose that $\calf_\bfX$ obeys an
Approximation Property and $\calg_\bfX$ obeys a Stable Interpolation
Property with respect to the family $\calf_\bfX$. Then a nearly
identical proof to the one for Proposition \ref{gaussian_pw_error} will
establish this generalized version of that proposition.

\begin{proposition}\label{abstract_error}
  Let $\Omega\subset \reals^n$ be a region satisfying a uniform cone
  condition. For $1\le p\le \infty$, let $\bfX$ be a collection of
  discrete subsets $\Xi \subset \Omega$ and let $\calf_\bfX$ be a
  family satisfying the Approximation Property.  If the interpolation
  operators $I_{\calg_\Xi}: C(\Omega)\to \mathcal{F}_{\Xi}$ satisfy the
  Stable Interpolation Property on $\calf_\bfX$, then for $k>n/p$,
$$
\|I_{\calg_\Xi} f- f\|_p \le C (1+ Q_p(\Xi)) h^k \|f\|_{W_p^k}\,,\
f\in W_p^k(\Omega).
$$
\end{proposition}

\section{Band-limited Interpolation and the Approximation
  Property}\label{bandlimited}

In this section, we will prove Lemma \ref{BandlimitedIntLemma}, which
asserts that interpolation by band-limited functions satisfies the Approximation
Property. 

\begin{proof}
  Our aim is prove the existence of the interpolant $\tilde f$, along
  with the other properties. Let $\widehat\varphi(\xi)\in C^\infty$ be
  such that
\begin{itemize}
\item $\text{supp } \widehat\varphi(\xi) \subset\{\xi\colon \
  |\xi|_\infty \le \pi+\vp\}$
\item 
$\widehat\varphi(\xi) = 1$ if $|\xi|_\infty\le \pi-\vp$
\item
$\sum\limits_{j\in {\cl H}^n} \widehat\varphi(\xi-2\pi j) = 1$.
\end{itemize}
Note that if $\widehat\varphi(\xi)$ satisfies the first two
conditions, then setting $\widehat\rho(\xi) =
\frac{\widehat\varphi(\xi)}{\sum\limits_{j\in{\ints}^n}
  \widehat\varphi(\xi - 2\pi j)}$ defines a function which satisfies
all three. Our candidate for $\tilde f$ is the function
\[
 g(x):= \sum_{j\in {\ints}^n} f(hj)\varphi\left(\frac{x}h-j\right).
\]
Clearly $g$ is in $\pw(b/h)$ and also $g|_{h\Z^n} = f|_{h\Z^n}$, thus
\eqref{eq1} and \eqref{eq2} are valid for $\tilde f=g$. To see
\eqref{eq3} and \eqref{eq4} also hold, write
\[
f = f_0+f_1\qquad
\text{where}\quad \hat f_0(\xi) := \hat f(\xi) \widehat\varphi(2h\xi). 
\]
Define $g = g_0+g_1$ accordingly, i.e.,
\[
g_0(x) := \sum_{j\in {\ints}^n} f_0(hj) \varphi\left(\frac{x}h-j\right),
\]
and $g_1:= g- g_0$.
Note that (for sufficiently small $\vp$)
\begin{align*}
\hat g_0(\xi) 
&= 
\left\{\sum_{j\in {\ints}^n} \hat f_0
\left(\xi - \frac{2\pi j}h\right)\right\} 
\widehat\varphi(h\xi)
= \hat f_0(\xi) \widehat\varphi(h\xi) = \hat f_0(\xi)
\end{align*}
so
\begin{equation}\label{eq5}
g_0(x) = f_0(x).
\end{equation}
The Sobolev seminorms of $g_0$ and $f_0$ are controlled by the seminorm of $f$:
\begin{equation}\label{eq6}
|f_0|_{W_p^k}
\le 
\sup_{|\alpha|=k} 
  \|D^{\alpha}(f*\rho)\|_{L_p} 
\le 
\|\rho\|_{L_1} |f|_{W_p^k},
\qquad 1\le p\le \infty
\end{equation}
where $\rho(x) = \frac1{(2h)^n} \phi(x/2h)$ and 
$ \|\rho\|_{L^1} = \|\phi\|_{L^1}$ is independent of $h$.
On the other hand, since 
$f_1= f - f*\rho $, 
we have,
\begin{equation}\label{eq7}
 \|f_1\|_{L^p} = \|f - (f*\rho)\|_p\le Ch^k |f|_{W_p^k},
\end{equation}
the usual error from  mollification by a band-limited mollifier.
What remains is to control the seminorm of $g_1$ and to bound the error 
$\|g_1 - f_1\|_p$.

For $|\alpha|=r$ with $0\le r\le k$, we have
\begin{align*}
D^{\alpha} g_1(x) 
&= 
\sum_{j\in {\ints^n}}  
  f_1(hj)D^{\alpha} 
  \left[\phi\left(\frac{x}h-j\right)\right]\\
&= 
h^{-r} 
\sum 
  f_1(aj) 
  [D^{\alpha} \phi] \left(\frac{x}h-j\right).
\end{align*}

When $p=\infty$, it follows that 
$$
\|D^{\alpha} g_1\|_{\infty}  
\le  
h^{-r} \sup_{j\in \ints^n} |f_1(hj)| 
\times 
\sup_{x\in \reals^n}\sum_{j\in\ints^n}
\left| 
  [D^{\alpha} \phi] 
  \left(\frac{x}h-j\right)
\right|
\le
C h^{-r} \sup_{j\in \ints^n} |f_1(hj)|,
$$
since the fact that $\phi$ is a Schwartz function implies that 
$\sup_{y\in \reals^n}\sum_{j\in\ints^n}
\left| [D^{\alpha} \phi] \left(y-j\right)\right|\le C$.

For $p<\infty$, we have
\begin{align*}
  |D^{\alpha}g_1(x)|^p &= (h^{-r})^p \left|\sum f_1(hj)
    [D^{\alpha}\phi]
    \left(\frac{x}h-j\right)\right|^p\\
  &\le h^{-rp} \left\{ \sum_j |f_1(h j)|^p \, \left| \left( 1 + \left|
          \frac{x}h-j \right| \right)^{2n} [D^{\alpha}\phi] \left(
        \frac{x}h-j \right) \right|^p
  \right\}\\
  &\quad \left.\times \left\{\sum_j \left(1 +
        \left|\frac{x}h-j\right|\right)^{-2np'}\right\}^{p/p'}\right. .
\end{align*}
Since $\phi$ is a Schwartz function, $\big\|\big(1+ \big|\frac{x}h
-j\big|\big)^{2n} [D^{\alpha}\phi]
\big(\frac{x}h-j\big)\big\|^p_{L^p(\reals^n)}$ is bounded by
$\widetilde C_rh^n$ where $h^n$ comes about from the substitution $y=
\frac{x}h-j$ in the integration over $\reals^n$.  Consequently,
\[
\|D^{\alpha}g_1\|_{L^p} \le h^{-r}\widetilde C_r\left\{\sum |f_1(hj)|^p h^n\right\}^{1/p}\times 
\max_{x \in \reals^n} \left\{\sum_j \left(1 + \left|\frac{x}h-j\right|\right)^{-2np'}\right\}^{p/p'}
\]
The sequence $j\mapsto  \big(1 + \big(\big| y-j\big|\big)^{-2n}$ 
is bounded in $\ell_p$, uniformly for all $y\in \reals^n$,
by a constant $C$. 
Thus, for $1\le p\le \infty$,
\begin{align}\label{eq8}
|g_1|_{W_p^r} 
&\le C_r h^{-r} h^{n/p} \|f_1|_{h\ints^n}\|_{\ell_p(\ints^n)}
\\
&\le C_r h^{-r} \left[\|f_1\|_{L_p} + h^k | f_1|_{W_p^k}\right].\notag
\end{align}
This follows by scaling the estimate 
$\|F|_{\ints^n}\|_{\ell_p(\ints^n)} 
\le 
C(\|F\|_{L_p(\reals^n)} + |F|_{W_p^k(\reals^n)})$, 
which holds for $k>n/p$ (this is
a simple consequence of the Sobolev embedding theorem).  
By applying \eqref{eq6} and \eqref{eq7}, we obtain
\begin{equation*}
|g_1|_{W_p^r}
\le  C h^{k-r} |f|_{W_p^k}.
\end{equation*}
Finally, we have
\begin{align*}
 |g|_{W_p^k} &= |g_0+g_1|_{W_p^k} = |f_0+g_1|_{W_p^k}
\le |f_0|_{W_p^k} + |g_1|_{W_p^k} \le {C} |f|_{W_p^k}
\end{align*}
which follows from \eqref{eq6} and above. So \eqref{eq3} is
established for $1\le p\le \infty$. Also from \eqref{eq5}, \eqref{eq6}
and \eqref{eq8} we get
\begin{align*}
\|f-g\|_{L^p} &= \|f_1-g_1\|_{L^p} \le \|f_1\|_{L^p} + \|g_1\|_{L^p}\\
&\le Ch^k |D^kf|_{L^p}.
\end{align*}
and \eqref{eq4} holds for $1\le p\le\infty$. Hence, it follows that
$g$ is indeed $\tilde f$.
\end{proof}
%
\section{Cardinal Interpolation with Gaussians}\label{CIG}
We now investigate the Fourier transform of the Lagrange function, and, through
this, the Fourier transform of the gridded Gaussian interpolant. Since 
$\widehat{\chi_h}(\xi) = 
h^n 
\widehat{L_{\lambda}^{[n]}}\bigl(h\cdot \xi \bigr)$ with
$\lambda= h^2$, and because 
$$
\widehat{L_{\lambda}^{[n]}}(\omega) 
= 
\frac{\exp\bigl(-|\xi|^2/(4\lambda)\bigr)}{\sum_{k\in \ints^n}\exp\bigl(-(|\xi-2\pi k|)^2/(4\lambda)\bigr)}
$$
we have
\begin{eqnarray} 
\widehat{\chi_h}(\xi) 
&=&
h^{n} \frac{\exp\bigl(
-\frac{1}{4}| \xi|^2\bigr)}{\sum_{k\in \ints^n}\exp\bigl(-\frac{1}{4}| \xi-\frac{2\pi k}{h}|^2\bigr)}  
=: h^n m_h^{[n]}(\xi).\label{mult_nef}
\end{eqnarray}
Throughout the rest of this article, the multiplier $m_h^{[n]}$ is the subject of much of our 
investigation. Its multiplier norm controls the the stability of the interpolation process
on Sobolev spaces, and ultimately, it determines the rate of decay of interpolation 
error. Because it is, evidently, a $n$-fold tensor product of univariate multipliers:
$m_h^{[n]}(\xi)  = m_h^{[1]}(\xi_1)  \cdots m_h^{[1]}(\xi_n)$, many of
our results can be obtained by considering the 1 dimensional case, where we suppress 
the dimension to write $m_h:=m_h^{[1]}.$

Finally, formula (\ref{mult_nef}) allows us to easily express the Fourier transform of the interpolant:
\begin{eqnarray}\label{IntFT}
\widehat{I_h f}(\xi) 
&=& 
\left[ \sum f(h j) \chi_h( \cdot - hj) \right]
\widehat{\phantom{\int} }(\xi)\nonumber 
=
h^n 
\left[  \sum_{j\in \ints^n} f(h j) e^{- i\langle hj,\xi\rangle} \right] 
m_h^{[n]}(\xi)\\
&=& 
\left[  
  \sum_{\beta \in 2\pi\ints^n} \widehat{f}\left(\xi - \frac{\beta}{h}\right) 
\right] 
m_h^{[n]}(\xi).
\end{eqnarray}
(The final equality follows from Poisson's summation formula.)
%
\subsection{The univariate multiplier $m_h$}

In the interest of keeping results self contained, we now provide some simple estimates
on the size of the multiplier. Such estimates (and stronger ones) could, with some effort, 
be gleaned from the work of Riemenschneider and Sivakumar. However, the ones we provide here are easier to obtain than those presented in 
\cite[Theorem 3.3 and Corollary 3.4]{RS2}, yet totally sufficient for our purposes.

We proceed in two stages. In the first stage we obtain estimates on $m_h^{\vee}$ that hold independently of $h$, but are rather slow (in particular, they do not show that
${m_h}^{\vee}$ is integrable). In the second stage, we demonstrate that ${m_h}^{\vee}(\xi)$ decays like $\mathcal{O}(|x|^{-2})$, and is, hence, integrable, but these estimates depend strongly on $h$.
 
{\bf First estimate of ${m_h}^{\vee}$:} Since $m_h(\xi)>0$, we have 
$|{m_h}^{\vee}(x)| \le {m_h}^{\vee}(0) =1$. On the other hand,
 from (\ref{mult_nef}),  
$$
m_h(\xi) 
= 
\left(\sum_{k\in \ints} e^{-\frac{\pi^2 k^2}{h^2}}e^{-\frac{\pi k \xi}{h}} \right)^{-1}
= 
\left(1+2\sum_{k=0}^{\infty} e^{-\frac{\pi^2 k^2}{h^2}}  \cosh(\pi k \xi /h) \right)^{-1}
=: \bigl(d_0(\xi)\bigr)^{-1}.
$$
Hence 
$
m_h'(\xi) 
= 
-\bigl(m_h(\xi)\bigr)^2 d_1(\xi)
$ 
and 
$m_h''(\xi) =  2 \bigl(m_h(\xi)\bigr)^3 \bigl(d_1(\xi)\bigr)^2 
- \bigl(m_h(\xi)\bigr)^2n_2(\xi)$,
where we have defined $d_1(\xi):= d_0'(\xi)$ and $d_2:=n_1'(\xi)$. Hence,
$$
d_1(\xi)= 
2
\sum_{k=1}^{\infty} 
  \left(
    \frac{\pi k}{h}\right) e^{-\frac{\pi^2 k^2}{h^2}} 
    \sinh\left(\frac{\pi k \xi}{h}
  \right)
\quad
\text{and}
\quad
d_2(\xi) =  
2
  \sum_{k=1}^{\infty} 
    \left(\frac{\pi k}{h}\right)^2 e^{-\frac{\pi^2 k^2}{h^2}}  
    \cosh(\frac{\pi k \xi}{h}).
$$
It follows that $m_h'(\xi)<0$ for positive $\xi$, and, by symmetry, $m_h'(\xi)>0$ for negative $\xi$. Therefore,
$\int_0^{\infty} |m_h'(\xi)|\dif \xi = - \int_0^{\infty} m_h'(\xi) \dif \xi = m_h(0)<1.$ 
It follows that
\begin{equation}\label{L1_mult_neriv}
\|m_h'\|_1<2,
\end{equation}
which implies that $|{m_h}^{\vee}(x)|< 2/|x|$. Thus,
\begin{equation}\label{first_mult_est}
|{m_h}^{\vee}(x)|\le \min(1, 2/|x|).
\end{equation}

{\bf Second estimate of ${m_h}^{\vee}$:}  
Using some simple algebraic manipulations, we may rewrite the second
derivative as 
\begin{equation}\label{second_neriv}
m_h''(\xi) =  m_h(\xi) 
\left[2 
  \left( 
  \frac{
    \sum_{k\in \ints} \left(\frac{k\pi}{h}\right) e^{-\frac{1}{4}|\xi - \frac{2\pi k}{h}|^2} 
  }{
    \sum_{k\in\ints} e^{-\frac{1}{4}|\xi - \frac{2\pi k}{h}|^2} 
  } \right)^2 
  -
    \left( 
  \frac{
    \sum_{k\in \ints} \left(\frac{k\pi}{h}\right)^2 e^{-\frac{1}{4}|\xi - \frac{2\pi k}{h}|^2} 
  }{
    \sum_{k\in\ints} e^{-\frac{1}{4}|\xi - \frac{2\pi k}{h}|^2} 
  } \right) 
\right]
\end{equation}
This leads us to the the following estimate:
\begin{lemma}\label{mult_second_neriv}
Let $h\le1$. There exists a constant $C$  so that
for $\tilde{k}=1,2, \dots$, and  $\xi \in [\frac{2\pi(\tilde{k}-1)}{h}, \frac{2\pi\tilde{k}}{h}]$,
$$|m_h''(\xi)|\le C \left(\frac{\tilde{k}}{h}\right)^2  m_h(\xi).$$
\end{lemma}
\begin{proof}
We write 
$$I:=  \left( 
  \frac{
    \sum_{k\in \ints} \left(\frac{k\pi}{h}\right) e^{-\frac{1}{4}|\xi - \frac{2\pi k}{h}|^2} 
  }{
    \sum_{k\in\ints} e^{-\frac{1}{4}|\xi - \frac{2\pi k}{h}|^2} 
  } \right)^2 
  \quad
  \text{and} 
  \quad
  II:= 
      \left( 
  \frac{
    \sum_{k\in \ints} \left(\frac{k\pi}{h}\right)^2 e^{-\frac{1}{4}|\xi - \frac{2\pi k}{h}|^2} 
  }{
    \sum_{k\in\ints} e^{-\frac{1}{4}|\xi - \frac{2\pi k}{h}|^2} 
  } \right) 
  $$
 We split the numerator of $I$ to isolate its two principal terms
$$
\left|
\sum_{k\in \ints} \left(\frac{k\pi}{h}\right) e^{-\frac{1}{4}|\xi - \frac{2\pi k}{h}|^2} 
\right|
\le
\left(\frac{\tilde{k}\pi}{h}\right)
\left(e^{-\frac{1}{4}|\xi- \frac{2\pi (\tilde{k}-1)}{h}|^2}
+ 
e^{-\frac{1}{4}|\xi -\frac{2\pi \tilde{k}}{h}|^2}\right)
+
2
\sum_{k=1}^{\infty} \left(\frac{k\pi}{h}+\frac{\tilde{k}\pi}{h}\right) e^{-| \frac{\pi k}{h}|^2}.
$$ 
Since it is a series of nonnegative terms, the denominator can be bounded below by its 
two largest summands:
$
\sum_{k\in\ints} e^{-\frac{1}{4}|\xi - \frac{2\pi k}{h}|^2} 
\ge 
e^{-\frac{1}{4}|\xi- \frac{2\pi (\tilde{k}-1)}{h}|^2}+ 
e^{-\frac{1}{4}|\xi -\frac{2\pi \tilde{k}}{h}|^2}
\ge 
2 e^{-\frac{\pi^2}{4h^2}}.
$
Therefore,
$$I\le   \left[   \left( 
  \frac{\tilde{k}\pi}{h}\right)+ 
  \sum_{k=1}^{\infty} \left(\frac{k\pi}{h}+\frac{\tilde{k}\pi}{h}\right) e^{-\bigl(| \frac{\pi k}{h}|^2-\frac{\pi^2}{4h^2}\bigr)}
  \right]^2.
  $$
 A similar argument shows that 
 $$
 II
 \le
 \left( 
  \frac{\tilde{k}\pi}{h}\right)^2+ 
 \sum_{k=1}^{\infty} \left(\frac{k\pi}{h}+\frac{\tilde{k}\pi}{h}
 \right)^2 
e^{-\bigl(| \frac{\pi k}{h}|^2-\frac{\pi^2}{4h^2}\bigr)}.
  $$
\end{proof}
An immediate consequence is that for $h\le 1$, $\|m_h''\|_{L_1} \le C h^{-3}$. 
Indeed, one
may estimate the integral on an interval around the origin $[-2\pi/h,2\pi,h]$ and
along the punctured real line
$\Omega :=\reals\setminus[-2\pi/h,2\pi/h]$ to obtain
$$\int_{-2\pi/h}^{2\pi/h}|m''(\xi)| \dif \xi\le C/h^3$$
and
$$\int_{\Omega}|m''(\xi)| \dif \xi\le 2\sum_{k=2}^{\infty} \frac{2\pi}{h} 
\|m''\|_{L_{\infty}(\left[\frac{2\pi(k-1)}{h}, \frac{2\pi k}{h}\right])}
\le 
2C\sum_{k=2}^{\infty} \frac{2\pi}{h} 
\left(\frac{k}{h}\right)^2  e^{-\bigl(| \frac{\pi k}{h}|^2-\frac{\pi^2}{4h^2}\bigr)}
\le
C.$$
It follows that
\begin{equation}\label{second_mult_est}
|{m_h}^{\vee}(x)| \le \frac{C}{h^3 |x|^2}.
\end{equation}
%
%
%
\subsection{The multiplier norm of $m_h^{[n]}$}
In the Section \ref{stable_interpolation} we investigate the Sobolev stability of Gaussian interpolation
on spaces of band-limited functions. Of particular importance are the operator norms
of the convolution operators with Fourier multiplier $m_h^{[n]}$.
These can be estimated in the case $p=1$ and $\infty$ 
by using the bounds on the cardinal interpolant obtained in \cite[Section 3]{RS2} (although the ones developed in the previous subsection, (\ref{first_mult_est}) and (\ref{second_mult_est}), suffice).
In case $1< p< \infty$, multiplier norms are estimated by appealing directly to the formula (\ref{mult_nef}), and making a comparison to the Hardy--Littlewood maximal function and the (maximal) Hilbert transform 
(this is a continuous version of an idea used in \cite[Theorem 3.1]{MRR}).
\begin{lemma}\label{multiplier_norm}
For $1< p < \infty$, there is a constant $C_p$ so that
$$
\|m_h^{[n]}\|_{\mult{p}}\le C_p .
$$
For $p=1,\infty$ there is a constant $C$ so that
$$
\|m_h^{[n]}\|_{\mult{1}}=\|m_h^{[n]}\|_{\mult{\infty}}\le C (1+| \log h|)^n.
$$
\end{lemma}
%
\begin{proof}
Because $m_h^{[n]}$ is a tensor product of univariate multipliers, it suffices to consider the case $n=1$. 

For $1/p+ 1/p' =1$, we have $\|m_h\|_{\mult{p}}=\|m_h\|_{\mult{p'}}$.
Thus,
$\|m_h\|_{\mult{1}}=\|m_h\|_{\mult{\infty}}\le \|{m_h}^{\vee}\|_1$, by H\"{o}lder's inequality.
>From (\ref{first_mult_est}) and (\ref{second_mult_est}), we have
\begin{eqnarray*}
\int_{\reals} |{m_h}^{\vee}(x)|\dif x 
&\le& 
2  \left[1  +2 \int_1^{h^{-3}} (x)^{-1} \dif x + Ch^{-1}\int_{h^{-3}}^{\infty} (h x)^{-2} \dif x \right]
\\
&=&
\left[C  + 12|\log h| \right]\le  C (1+|\log h|)
\end{eqnarray*}

We now consider the case $1<p<\infty$. 
Let $f\in L_p$ and let $g\in L_{p'}$  with $\|f\|_{p} = 1 = \|g\|_{p'}$.
We can estimate $\|m_h\|_{\mult{p}}$ by the supremum of the expression
$\left|\int_{\reals}\int_{\reals} f(x) {m_h}^{\vee}(t-x) g(t) \dif x \dif t \right|$. To this end, let
$\Omega_h(t):=\reals\setminus (t-h,t+h)$. 
Then, 
\begin{eqnarray*}
\left|\int_{\reals}\int_{\reals} f(x) {m_h}^{\vee}(t-x) g(t) \dif x \dif t \right|
&\le &
\left|\int_{\reals}\int_{t-h}^{t+h} f(x) {m_h}^{\vee}(t-x) g(t) \dif x \dif t \right|\\
&\mbox{}&+ 
\left|\int_{\reals}\int_{\Omega_h(t)} f(x) {m_h}^{\vee}(t-x) g(t) \dif x \dif t \right|\\
&=:&I+II
\end{eqnarray*}
The first integral 
$I\le\int_{\reals}h^{-1}\int_{t-h}^{t+h}|f(x)| |g(t)| \dif x \dif t ,$ 
can be compared to an integral involving the maximal function of $f$, 
$f^{\sharp}(t) := \sup_{\epsilon>0} (2\epsilon)^{-1}
\int_{t-\epsilon}^{t+\epsilon} |f(x)|\dif x$. Thus,
$$
I \le 2\int_{\reals}  f^{\sharp}(t)|g(t)|\dif t
\le 
2\|f^{\sharp}\|_p \|g\|_{p'} 
\le 
2C_p\|f\|_p\|g\|_{p'}.$$

To treat $II$, we note that, 
$m_h'(\xi)$ is integrable, and
$II$ can be estimated by employing the
Fourier transform of ${m_h}^{\vee}$, 
(this is a trick similar to the one used in \cite{MRR}):
\begin{eqnarray*}
II &=&
\left|
  \int_{\reals}\int_{\Omega_h(t)} \int_{\reals}
      f(x) m_h (\xi) e^{i (t-x)\xi} g(t)
  \dif \xi \dif x \dif t 
\right|\\
&=&
\left|
  \int_{\reals}\int_{\Omega_h(t)} \int_{\reals}
      \frac{f(x)}{i(x-t)} m_h' (\xi) e^{i (t-x)\xi} g(t)
  \dif \xi \dif x \dif t 
\right|\\
&=&
\left|
   \int_{\reals}  \ m_h' (\xi) \int_{\reals}\int_{\Omega_h(t)} 
      \frac{f(x)}{i(x-t)} e^{i (t-x)\xi} g(t)
   \dif x \dif t   \dif \xi.
\right|
\end{eqnarray*}
The second equality follows by integration by parts, while the third follows by Fubini's 
theorem, since $m_h'$  is integrable on $\reals$ and $ \frac{f(x)}{(x-t)} g(t){\1}_{ \Omega_t}(x)$ is integrable on $\reals^2$. It follows that
\begin{eqnarray*}
II &\le& \int_{\reals} \left|  \ m_h' (\xi) \right| 
     \int_{\reals}
        \left|\int_{\Omega_h(t)}\frac{e^{-i\xi x}f(x)}{x-t}\dif t \right|  |g(t)|
 \dif t  \dif \xi\\
&\le& \int_{\reals} \left|   m_h' (\xi) \right| 
   \int_{\reals}
      \bigl[\mathbf{H}(e^{-i\xi \cdot}f)\bigr] (t)\ |g(t)|\,
 \dif t   \dif \xi\\
  &\le& C_p \|m_h'\|_1\|f\|_p\|g\|_{p'}\le 2C_p \|f\|_p\|g\|_{p'}.
\end{eqnarray*}
In the second inequality, we use the {\em maximal Hilbert transform}  
$ \bigl[\mathbf{H}F\bigr] (x) := \sup_{\epsilon>0}|\int_{\Omega_{\epsilon}(t)}F(t) \frac{\dif t}{x-t}|$.
The third inequality follows from the fact that the maximal Hilbert transform is
strong type $(p,p)$ (i.e., $\|\mathbf{H}F\|_p\le C_p \|F\|_p$) 
for $1<p<\infty$, as observed in \cite[Theorem 4.9]{BeSh}. 
The fourth inequality follows from (\ref{L1_mult_neriv}).
\end{proof}

\section{Stable interpolation of band-limited functions by
  Gaussians}\label{stable_interpolation}
We consider in this section interpolation of functions in
$\pw(b/h)$ 
and show that Gaussian interpolation restricted to such functions is
stable with respect to each Sobolev norm $W_p^k$, with $k>n/p$.

Before stating and proving the result in its full generality, we indicate how the proof works
in the univariate case.
We use (\ref{IntFT}) to write, for a nonnegative integer $ k$,
\begin{eqnarray*}
\widehat{ (I_h f)^{(j)}}(\xi) &=& \xi^{j} 
\left[ \sum_{\beta\in 2\pi \ints} \widehat{f}(\xi - \frac{\beta}{h})\right] m_h(\xi)
=
\sum_{\beta \in 2\pi \ints} \widehat{G_{j,\beta}}(\xi) 
\end{eqnarray*}
where
$
G_{j,\beta} 
:= 
D^{j} \left[(e^{i  (\cdot) \beta/h)}f)* (m_h)^{\vee}\right].
$
Clearly, 
$\|I_h f\|_{W_p^k} \le 
\sum_{j=0}^{ k}\sum_{\beta}\|G_{j,\beta}\|_p.$ 

The result we are after requires estimating $\|G_{j,\beta}\|_p$ for various values of $\beta$.
These estimates fall into three
categories: $\beta = 0$, $|\beta|=2\pi$, and $|\beta|>2\pi$.

{\bf Estimating $\|G_{j, 0}\|_p$:} 
In this case,
$
 \left\| D^{j}   f * (m_h)^{\vee} \right\|_p
 \le 
 \|D^{j}f\|_{p} \|m_h\|_{ \mult{p}}
$
from which we obtain
$
\sum_{j\le k}\|G_{j,0}\|_p 
\le 
C  \|f\|_{W_p^k} \|m_{h}\|_{\mult{p}}$. 
Thus, from Lemma \ref{multiplier_norm},
\begin{equation}\label{=0}
\sum_{j\le k}\|G_{j,0}\|_p 
\le 
\begin{cases}
C_p   \|f\|_{W_p^k}&\quad 1<p<\infty\\
C  (1+|\log h|)\|f\|_{W_p^k}& \quad p=1,\infty 
\end{cases}  
\end{equation}
where $C$ and $C_p$ depend only on $k$ and $p$.

For some of the other terms, we need to use a special cutoff function.  We first consider a smooth bump function,
$\upsilon$. It is a non-negative $C^{\infty}$ function satisfying
\begin{itemize}
\item $\upsilon(\xi)=0$ if $|\xi|>2\epsilon$ 
and 
$\upsilon(\xi)=1$ if $|\xi|<\epsilon$,
\item $\upsilon(-\xi) = \upsilon(\xi)$.
\end{itemize}
We use $\upsilon$ to construct univariate cutoff functions $\varphi$ having support in $[-\pi-2\epsilon,\pi+2\epsilon]$:
\begin{itemize}
\item $\varphi(t) =1$ for $-\pi \le t\le\pi$;
\item $\varphi(-t) = \varphi(t)  = \upsilon(t-\pi)$ for $t\ge \pi$.
\end{itemize}

{\bf Estimating $\|G_{j, \beta}\|_p$, for $|\beta|>2\pi$:} 
In this case, the fact that $f$ is band limited, $\mathrm{supp}\,\hat{f}\subset B(0,(\pi+\epsilon)/h)$, 
allows us to write
$$
\widehat{G_{j,\beta}}(\xi)
= \xi^{j}\widehat{f}(\xi - \frac{\beta}{h})m_h(\xi) 
=  \widehat{f}(\xi - \frac{\beta}{h})  \xi^{j} \varphi\bigl(h ( \xi - \frac{\beta}{h}) \bigr)m_h(\xi).
$$
The norm of the multiplier  $\tau_1:=\tau_{1,j,\beta}(\xi):=\xi^{j} \varphi\bigl(h ( \xi - \frac{\beta}{h}) \bigr)m_h(\xi)$ can be estimated by $\|\tau_1\|_1$, which we can estimate using the fact that
$$
\max\left(
  \int_{\reals} 
   \left |\frac{d^2\widehat{g}}{d\xi^2}  (\xi)\right| 
  \dif \xi,
   \int_{\reals} 
   \left |\widehat{g}  (\xi)\right| 
  \dif \xi
 \right)
\le 
K 
\quad \Longrightarrow \quad
|g(x)| 
\le 
\frac{K}{(1+|x|)^{2}} 
\quad \Longrightarrow \quad
\|g\|_1 \le CK.$$
By (\ref{first_final_estimate}) of Lemma \ref{mu_L1}, we have the following.
%
\begin{claim}\label{tau1}
For $\beta\in 2\pi\ints^n$, $|\beta|>2\pi$,
$\|\tau_{1,j,\beta}\|_1 \le (C/h) |\beta/h|^{(2+j)} \exp(-c|\beta|^2/h^2)$. 
\end{claim}
%
Therefore,
\begin{equation}\label{>2pi}
\|G_{j,\beta}\|_p \le 
\|\tau_1\|_1 \left\|\left(\widehat{f}(\xi - \frac{\beta}{h})\right)^{\vee}\right\|_p
\le
\left(\frac{C}{h}\right) \left|\frac{\beta}{h}\right|^{(2+j)} \exp\left(-c\frac{|\beta|^2}{h^2}\right)\|f\|_p.
\end{equation}

{\bf Estimating $\|G_{j, \beta}\|_p$, for $|\beta|=2\pi$:} 
This remaining case is quite similar to the previous one. 
We again use the fact that $f$ is band limited, although we need to
exercise extra caution since  the cutoff $\varphi(h(\xi- \beta/h))$ 
overlaps a narrow  region (near to $\beta/2h$) 
where $m_h(\xi)$ is close to $1$ and
$|\xi^{j}|$ is very large. 

When
$\beta=2\pi$ we write 
$ \varphi(t) = \omega(t) + \upsilon(t +\pi)$, with $\omega$ having
support on the (non-symmetric interval) $[-\pi +\epsilon, \pi+2\epsilon]$ and
equaling $1$ on $[-\pi +2\epsilon, \pi +\epsilon]$. (When $\beta =-2\pi$ an obvious
modification  $\varphi(t) = \tilde{\omega}(t) + \upsilon(t +\pi)$ is made.)
This allows us to write
\begin{eqnarray}
\widehat{G_{j,\beta}}(\xi)
&=&  \widehat{f}(\xi - \frac{\beta}{h})  \xi^{j}  \upsilon\bigl(h (\xi- \frac{\pi}{h}) \bigr)m_h(\xi)
+
\widehat{f}(\xi - \frac{\beta}{h})  \xi^{j}  \omega(h(\xi-2\pi/h)) m_h(\xi)\nonumber\\
&=: & \widehat{f}(\xi - \frac{\beta}{h})  \tau_{2,j,\beta}(\xi)
+
\widehat{f}(\xi - \frac{\beta}{h})  \tau_{3,j,\beta}(\xi).\label{tau2and3}
\end{eqnarray}

We first investigate $\tau_2=\tau_{2,j,\beta}$ by rewriting the monomial 
 $\xi^{j}$ as a Taylor series about $2\pi/h$, obtaining
\begin{multline*} 
 \tau_2(\xi)=\xi^{j} \upsilon\bigl(h (\xi - \frac{\pi}{h}) \bigr)m_h(\xi) 
 = 
\sum_{\ell=0}^{j}  \binom{j}{\ell} \left(\frac{2\pi}{h}\right)^{j-\ell}
\left( \xi-\frac{2\pi}{h}\right)^{\ell}
\upsilon\bigl(h ( \xi - \frac{\pi}{h}) \bigr)m_h(\xi) \\
=
\left[\left( \xi-\frac{2\pi}{h}\right)^{j} \right]
\times 
m_h(\xi)
\times 
\left[\sum_{\ell=0}^{j} \binom{j}{\ell} \left(\frac{2\pi}{h}\right)^{j-\ell}
\frac{\upsilon\bigl(h ( \xi - \frac{\pi}{h}) \bigr)}{\left( \xi-\frac{2\pi}{h}\right)^{j-\ell}}\right].
\end{multline*}
The multiplier norm of
$\mu(\xi):= \sum_{\ell=0}^j \binom{j}{\ell} \left(\frac{2\pi}{h}\right)^{j-\ell}\frac{\upsilon\bigl(h ( \xi - \frac{\pi}{h}) \bigr)}{\left( \xi-\frac{2\pi}{h}\right)^{j-\ell}}
=
 \sum_{\ell=0}^j \binom{j}{\ell} \left({2\pi}\right)^{j-\ell}\frac{\upsilon\bigl(h  \xi - {\pi} \bigr)}{\left( h\xi- {2\pi}\right)^{j-\ell}}$ 
(which is a Schwarz function, since the support of $\upsilon\bigl(\cdot - {\pi} \bigr)$ is positive distance
from $2\pi$)
is uniformly bounded (in $h$) by the constant
$$\mathcal{K}_j := \sum_{\ell=0}^j \binom{j}{\ell} \left({2\pi}\right)^{j-\ell} 
\left\|\left(      \frac{\upsilon\bigl(  \cdot - {\pi}) \bigr)}{\left( \cdot - {2\pi}\right)^{j-\ell}}      \right)^{\vee}\right\|_1.$$
Thus we have shown the following:
\begin{claim}\label{tau2}
The multiplier $\tau_{2,j,\beta}(\xi)$ can be written $\tau_{2,j,\beta}(\xi) = (\xi-\beta/h)^j \times m_h(\xi)\times \mu(\xi),$
where  $\mu$ has multiplier norm $\|\mu\|_{\mult{p}}\le \mathcal{K}_j$, bounded independent of $h$.
\end{claim}

The multiplier $\tau_{3,j,\beta}$ is controlled in a similar way to $\tau_1$.
The only modification is that we use estimate (\ref{second_final_estimate})
of Lemma \ref{mu_L1}, utilizing the fact that $\omega(\cdot - 2\pi)$ has support a positive distance (namely $\epsilon$)
from $\pi$. In particular, it satisfies condition (\ref{extra_condition}) Thus, we  obtain
%
\begin{equation*}
\|(\tau_{3})^{\vee}\|_1\le
 \left|\frac{C}{h}\right|^{{j}+3} e^{-|c|^2/h^2},
 \end{equation*}
 and we have demonstrated the following claim.
 \begin{claim}\label{tau3} Let $|\beta|=2\pi$.
 There is a constant $C$ depending only on $j$ for which the multiplier $\tau_{3,j,\beta}$ satisfies $$\|(\tau_{3,j,\beta})^{\vee}\|_1\le C.$$
\end{claim}

For $|\beta|=2\pi$, applying the Claims \ref{tau2} and \ref{tau3} to (\ref{tau2and3}), it follows that  
\begin{eqnarray}
\|G_{j,\beta}\|_p&\le& \|(\tau_2 \widehat{f}(\cdot-\beta/h))^{\vee}\|_p + 
\|\tau_3 \widehat{f}(\cdot-\beta/h))^{\vee}\|_p\nonumber\\
&\le & 
\mathcal{K}_j \|m_h\|_{\mult{p}}
 \left\| \left( (\cdot - \beta/h)^j \widehat{f}(\cdot-\beta/h)\right)^{\vee}\right\|_p + 
C\left\|\left(\widehat{f}(\cdot-\beta/h)\right)^{\vee}\right\|_p\nonumber\\
&\le &
\mathcal{K}_j  \|m_h\|_{\mult{p}} \| f\|_{W_p^j} + 
C\|f\|_p\label{=2pi}
\end{eqnarray}
Summing $\|G_{j,\beta}\|_p$ over $0\le j\le k$ and $\beta \in 2\pi
\ints$, and employing the estimates (\ref{=0}), (\ref{>2pi}) and
(\ref{=2pi}), we observe that $\|I_h f\|_{W_p^k} \le C(1 +
\|m_h\|_{\mult{p}})\|f\|_{W_p^k}$ for $f\in \pw((\pi+\epsilon)/h)$.

We now give the general, multivariate result.
\begin{lemma}\label{StableIntLemma}
  Let $0<\epsilon<\pi/2$.  Cardinal interpolation by Gaussians on
  $h{\ints^n}$ restricted to band-limited functions in
  $\pw((\pi+\epsilon)/h)$, satisfies the Stable Interpolation
  Property, with
\begin{itemize}
\item $Q_p(h)\le C_p$, a constant depending only on $\epsilon,$
  $p$, $n$ and $k$ when $1<p<\infty$,
\item $Q_p(h)\le C(1+|\log h|)^n$, a constant depending only on
  $\epsilon,$ $n$ and $k$ when $p=1,\infty$,
\end{itemize}
In other words,
$$
|I_h f|_{W_p^k} \le Q_p(h) \|f\|_{W_p^k}\qquad \text{for}\ f\in
\wp(\frac{\pi+\epsilon}{h})
$$
\end{lemma}
%
\begin{proof}
We use (\ref{IntFT}) to write, for a multi-integer $|\alpha|\le k$,
\begin{eqnarray*}
\widehat{D^{\alpha} I_h f}(\xi) &=& \xi^{\alpha} \left[ \sum_{\beta\in 2\pi\ints^n} \widehat{f}(\xi - \frac{\beta}{h})\right] m_h^{[n]}(\xi)
=
\sum_{\beta \in 2\pi \ints^n} \widehat{G_{\alpha,\beta}}(\xi) 
\end{eqnarray*}
where
$
G_{\alpha,\beta} 
:= 
D^{\alpha} \left[(f e^{i \langle \beta/h, \cdot \rangle})* (m_h^{[n]})^{\vee}\right].
$
Clearly, 
$\|I_h f\|_{W_p^k} \le 
\sum_{|\alpha|\le k}\sum_{\beta}\|G_{\alpha,\beta}\|_p,
$ and the remainder of the section is concerned with estimating 
$\|G_{\alpha,\beta}\|_p$ for various $\beta$. 


We write $\widehat{G_{\beta,\alpha}} (\xi)= M_{\alpha}(\xi)\widehat{f}(\xi- \beta/h)$
employing the tensor product multiplier $M_{\alpha} (\xi)
:=
\xi^{\alpha} 
m_{h}^{[n]}(\xi) $.
We estimate these terms by observing that the support of $\widehat{f}(\xi-\beta/h)$,
which is a neighborhood of $\beta/h$, and therefore lies in a region where $\mu_{\alpha}$ is small
(for most values of $\beta \ne 0$). 

This heuristic is complicated by the fact that, for certain values of $\beta$, the neighborhood
of $\beta/h$ overlaps a region where $m_h^{[n]}$ is near to $1$ and $|\xi^{\alpha}|$ may be quite
large. Therefore, we must be somewhat  careful.

Fix $\beta = (\beta_1,\dots,\beta_n) \in 2\pi \ints^n$. Before proceeding, we partition 
$(1,\dots,n)$ into three subsequences 
$(J_1,\dots,J_{n_1})$, 
$(K_1,\dots,K_{n_2})$ and 
$(L_1,\dots,L_{n_3})$ , with $n_1+n_2+n_3=n$, and 
where   
\begin{itemize}
\item $(J_1,\dots, J_{n_1})$ are the indices $J$ where $|\beta_{J} |>2\pi$
\item $(K_1,\dots,K_{n_2}) = \supp{\beta}$ , 
\item $(L_1,\dots,L_{n_3})$ are the indices  where $|\beta_L|=2\pi.$
\end{itemize}
As an example, in dimension $n=5$, we have for $\beta = (-4\pi, -2\pi, 0, 0, 6\pi)$
that $(J_1,J_2) = (1,5)$, $(K_1,K_2)=(3,4)$ and $L_1= 2$.

Then, because the multiplier $m_h^{[n]}$ and the monomial $\xi^{\alpha}$ 
are tensor products and can be factored over the three
subsequences we have just constructed, and because because $f$ is band limited, we have 
\begin{multline*}
M_{\alpha} (\xi) \widehat{f}(\xi- \frac{\beta}{h})
=
\left[
\prod_{j=1}^{n_1}
(\xi_{J_j})^{\alpha_{J_j}}
 \varphi\bigl(h ( \xi_{J_j} - \frac{\beta_{J_j}}{h}) \bigr)
m_{h}(\xi_{J_j})\right]\\
\times
\left[
\prod_{j=1}^{n_2}
(\xi_{K_j})^{\alpha_{K_j}} 
m_{h}(\xi_{K_j}) \right]
\\
\times
\left[
\prod_{j=1}^{n_3}
(\xi_{L_j})^{\alpha_{L_j}}
 \varphi\bigl(h ( \xi_{L_j} - \frac{\beta_{L_j}}{h}) \bigr)
m_{h}(\xi_{L_n})\right]
\times
\widehat{f}(\xi- \frac{\beta}{h})
\end{multline*}
In other words, 
$M_{\alpha} (\xi) \widehat{f}(\xi- \frac{\beta}{h})$ can be written as a product of tensor product multipliers applied to $\widehat{f}(\xi- \frac{\beta}{h})$, namely
$$
M_{\alpha} (\xi) \widehat{f}(\xi- \frac{\beta}{h})=
[\sig_1(\xi)]
\times
\left[\prod_{j=1}^{n_2}(\xi_{K_j})^{\alpha_{K_j}} 
m_{h}(\xi_{K_j})\right]
\times
[\sig_2(\xi)]\times  \widehat{f}(\xi- \frac{\beta}{h}),$$
We have written
$\sig_1(\xi):= \sig_{1,\alpha,\beta}(\xi)
:=\prod_{j=1}^{n_1}
(\xi_{J_j})^{\alpha_{J_j}}
 \varphi\bigl(h ( \xi_{J_j} - \frac{\beta_{J_j}}{h}) \bigr)
m_{h}(\xi_{J_j})
=\prod_{j=1}^{n_1}\tau_1(\xi_{J_j})
$.
In a similar way, we identify the factor 
where $\beta = 0$ as
$\prod_{j=1}^{n_2}
(\xi_{K_j})^{\alpha_{K_j}} 
m_{h}(\xi_{K_j})
$
and the factor
where $|\beta|=2\pi$ as
$\sig_2(\xi):= \sig_{2,\alpha,\beta}(\xi):=\prod_{j=1}^{n_3}
(\xi_{L_j})^{\alpha_{L_j}}
 \varphi\bigl(h ( \xi_{L_j} - \frac{\beta_{L_j}}{h}) \bigr)
m_{h}(\xi_{L_j}).
$
It follows that 
\begin{eqnarray}
\|G_{\alpha,\beta}\|_p 
&=&
 \left\| \left(M_{\alpha}\ \widehat{f}(\cdot- \beta/h)\right)^{\vee}\right\|_p \nonumber\\
&\le& 
\|\sig_{1,\alpha,\beta}\|_{\mult{p}}
\times \|m_h\|_{\mult{p}}^{n_2} 
\times 
\left \| 
\left(
  \sig_{2,\alpha,\beta} 
\times (\cdot)^{\alpha_K}\widehat{f}(\cdot - \beta/h) \right)^{\vee}\right \|_{L_p}.\label{G_est}
\end{eqnarray} 
where we write $(\cdot)^{\alpha_K}:\xi\mapsto \xi^{\alpha_{K}}  := \prod_{j=1}^{n_2}\xi^{\alpha_{K_j}}.$

{\bf Estimating $\|(\sig_{1,\alpha,\beta})^{\vee}\|_{\mult{p}}$:}
We use the rough estimate
 $\|(\sig_{1,\alpha,\beta})^{\vee}\|_{\mult{p}}\le \|(\sig_{1,\alpha,\beta})^{\vee}\|_1$. The $L_1$ norm of each univariate factor 
$ \tau_{1}(\xi_{J_j}):=
(\xi_{J_j})^{\alpha_{J_j}}
 \varphi\bigl(h ( \xi_{J_j} - \frac{\beta_{J_j}}{h}) \bigr)
m_{h}(\xi_{J_j})$ is bounded by Claim \ref{tau1}, from which we obtain
\begin{equation}\label{Tau_1}\|
(\sig_{1})^{\vee}\|_1\le
 \frac{C}{h^{n_1}} \prod_{j=1}^{n_1}\left|\frac{\beta_{J_j}}{h}\right|^{\alpha_{J_j}+2}e^{-c|\beta_{J_j}|^2/h^2}.
 \end{equation}
{\bf Estimating 
$\left \| 
\left(
  \sig_{2,\alpha,\beta} 
\times (\cdot)^{\alpha_K} \widehat{f}(\cdot - \beta/h) \right)^{\vee}\right \|_{L_p}$:}
 There is an immediate decomposition of 
 $
  \sig_2(\xi)$
into
$$
\prod_{j=1}^{n_3}
\left( 
\tau_2(\xi_{L_j})+\tau_3(\xi_{L_j})
\right). 
$$
Thus Claims \ref{tau2} and \ref{tau3}, and the fact that
$\left\|\left( \prod_{j=1}^{n_3} (\xi_{L_j} - \beta_{L_j})^{\alpha_{L_j}} \times \xi^{\alpha_{K}}\times \widehat{f}(\xi - \beta/h)\right)^{\vee}\right\|_p\le \|f\|_{W_p^k}$
 imply that
\begin{equation}\label{Tau_3_split}
\left \| 
\left(
  \sig_{2} 
\times (\cdot)^{\alpha_{K}}\widehat{f}(\cdot - \beta/h) \right)^{\vee}\right \|_{L_p}
\le
C\left[\prod_{j=1}^{n_3}
(1+\|m_h\|_{\mult{p}})
\right]
\left \|  f \right \|_{W_p^k}.
\end{equation}
with constant $C$ depending on $n$ and $\alpha$.

Applying  estimates (\ref{Tau_1})  and (\ref{Tau_3_split}), which control $\|(\sig_{1,\alpha,\beta})^{\vee}\|_{\mult{p}}$, and 
 $\left \| \left(
  \sig_{2,\alpha,\beta} 
\times (\cdot)^{\alpha_K} \widehat{f}(\cdot - \beta/h) \right)^{\vee}\right \|_{L_p}$, respectively,
to  (\ref{G_est}),  we bound the sum of the $\|G_{\alpha,\beta}\|_p$'s:
\begin{equation}
\label{far-off-center}
\sum_{\beta\in 2\pi \ints^n } \|G_{\alpha,\beta}\|_p
\le
C(1+\|m_h\|_{\mult{p}})^n \|f\|_{W_p^k}
\left[\prod_{j=1}^n
\left(
3+ 2\sum_{\ell = 2}^{\infty}
 \left|\frac{2\pi \ell}{h}\right|^{\alpha_j+2}%
e^{-c\left|\frac{ 2\pi\ell}{h}\right|^2}  
\right) 
\right]
\end{equation}
and the Lemma follows from Lemma \ref{multiplier_norm}.

\end{proof}
\begin{lemma}\label{mu_L1}
Suppose that  $\epsilon>0$ and that $\phi$ is a $C^{\infty}$ function with support in 
$B(0,\pi+\epsilon)$.
If $\beta \in 2\pi \ints$, $\beta>2\pi$ and $0<h<1$
then there exist constants $c,C>0$, depending only on $\epsilon$ and $k$  
so that
\begin{equation}\label{first_final_estimate}
\int_{\reals}
  \left|
    \frac{d^2}{d\xi^2}\left[
      \xi^{k} 
      \phi \bigl(h ( \xi - \frac{\beta}{h})\bigr)
      m_{h}(\xi)
    \right] 
  \right|
\dif \xi 
\le  
C h^{-1}
  \left(\frac{|\beta|}{h}
\right)^{2+k} 
\exp{\left(-c\left|\frac{\beta}{h}\right|^2\right)}.
\end{equation}
For $\beta \in 2\pi \ints$, $|\beta|=2\pi$ and  $\phi$ satisfying the extra condition
\begin{equation}\label{extra_condition}
\supp{\phi(\cdot-\beta)}\cap B(0, \pi+\epsilon) = \emptyset,
\end{equation}
there exist constants $c,C>0$, depending only on $\epsilon$ and $k$  
so that for $0<h<1$
\begin{equation}\label{second_final_estimate}
\int_{\reals}
  \left|
    \frac{d^2}{d\xi^2}\left[
      \xi^{k} 
      \phi\bigl(h ( \xi - \frac{\beta}{h})\bigr)
      m_{h}(\xi)
    \right] 
  \right|
\dif \xi 
\le  
C h^{-(3+k)} \exp{\left(-\left|\frac{c}{h}\right|^2\right)}.
\end{equation}
\end{lemma}
\begin{proof}
Let $b=\pi + 2\epsilon$.
We prove this by making the estimate
$$
\int_{\reals}
  \left|
    \frac{d^2}{d\xi^2}\left[
      \xi^{k} 
      \phi\bigl(h ( \xi - \frac{\beta}{h})\bigr)
      m_{h}(\xi)
    \right] 
  \right|
\dif \xi 
\le 
\left(\frac{2b}{h}\right) 
\max_{\xi\in \left[\frac{\beta-b}{h},\frac{\beta+b}{h}\right]} 
\left|\frac{d^2}{d\xi^2}\left[\xi^{k} 
\phi\bigl(h ( \xi - \frac{\beta}{h}) \bigr)
m_{h}(\xi)\right] \right|.$$

By applying the product rule to the expression being maximized, we obtain
$$
 \frac{d^2}{d\xi^2}
\left[
  \xi^{k} 
  \phi\bigl(h ( \xi - \frac{\beta}{h}) \bigr)
  m_{h}(\xi)
\right] 
=
\sum_{
|\gamma| =2}
C_{\gamma}
\times
\left(D^{\gamma_1}\xi^{k}\right)
\times 
\left(D^{\gamma_2} \phi\bigl(h ( \xi - \frac{\beta}{h}) \bigr)\right)
\times 
\left(D^{\gamma_3}m_{h}(\xi)\right).
$$
with 
$C_{\gamma} = \frac{2!}{\gamma_1! \gamma_2!\gamma_3!}$.

Without loss, it suffices to consider only components where both derivatives
are on the third factor, i.e., those of the form
$
\xi^{k} \phi\bigl(h ( \xi - \frac{\beta}{h}) \bigr)\frac{d^2}{d\xi^2}\left[m_{h}(\xi)\right],
$
since the other terms 
are small compared to these.
Indeed,
$
|D^{\gamma} \xi^{k}| 
= C |\xi^{k-\gamma}|
\le C |\beta/h|^{k}
$
for $\xi \in [(|\beta|- b)/h, (|\beta|+ b)/h]$
(with $C$  depending only on $k$),
and 
$
\max_{\xi\in \left[\frac{\beta-b}{h},\frac{\beta+b}{h}\right]} 
\left|D^{\gamma} \phi\bigl(h ( \xi - \frac{\beta}{h}) \bigr)\right| 
= 
h^{|\gamma|} 
\max_{\xi\in \left[\frac{\beta-b}{h},\frac{\beta+b}{h}\right]}
\left|\phi\bigl(h ( \xi - \frac{\beta}{h}) \bigr)\right|
\le Ch^{|\gamma|},$
since $\tau := D^{\gamma} \phi$ is a $C^{\infty}$ function with the same support as $\phi$. 

Hence,
$$
\max_{\xi\in \left[\frac{\beta-b}{h},\frac{\beta+b}{h}\right]} 
\left|    \frac{d^2}{d\xi^2}   \left[\xi^{k} 
\phi\bigl(h ( \xi - \frac{\beta}{h}) \bigr)
m_{h}(\xi)\right] \right|
\le 
C
\left|\frac{\beta}{h}\right|^{k} 
\max_{h\xi\in \supp{\phi(\cdot - \beta)}} \left|  m_{h}''(\xi) \right|.
$$
The result now follows directly from Lemma \ref{mult_second_neriv}. Indeed, we have:
$$
\max_{h\xi\in \supp{\phi(\cdot - \beta)}} \left|  m_{h}''(\xi) \right|
\le 
C \left(\frac{|\beta|+b}{h}\right)^2  
\max_{h\xi\in \supp{\phi(\cdot - \beta)}}m_h(\xi)
\le C \left(\frac{|\beta|+b}{h}\right)^2  \max_{h\xi\in \supp{\phi(\cdot - \beta)}}  \frac{e^{-\frac{1}{4}|\xi|^2}}{2e^{-\frac{1}{4} \left|\frac{\pi}{h}\right|^2}}.$$
If (\ref{extra_condition}) holds, then the expression being maximized can be controlled
by 
$\exp(-\frac{1}{4}\frac{\left(|\pi + \epsilon|^2-\pi^2\right))}{h^2}$ and (\ref{second_final_estimate}) follows.

On the other hand, if $|\beta|\ge 4\pi$, then 
$h\xi \in  \supp{\phi(\cdot - \beta)}$ implies that 
$|\xi|^2\le \frac{(|\beta|-b)^2}{h^2}$ and 
$  \exp(-\frac{1}{4}|\xi|^2)  \exp(\frac{1}{4} \left|\frac{\pi}{h}\right|^2) \le 
\exp(-c\frac{|\beta|^2}{h^2})$, from which (\ref{first_final_estimate}) follows.
\end{proof}

\bibliographystyle{siam}
\bibliography{cardinal}

\end{document}